\documentclass[11pt]{article}
\usepackage{amsfonts, amssymb, amsmath, amsthm, xcolor, graphicx}
\usepackage[colorlinks=true,citecolor=blue,linkcolor=blue, urlcolor=blue]{hyperref}

\begin{document}

\theoremstyle{plain}
\newtheorem{thm}{Theorem}[section]
\newtheorem{cor}[thm]{Corollary}
\newtheorem{con}[thm]{Conjecture}
\newtheorem{cla}[thm]{Claim}
\newtheorem{lm}[thm]{Lemma}
\newtheorem{prop}[thm]{Proposition}
\newtheorem{example}[thm]{Example}

\theoremstyle{definition}
\newtheorem{dfn}[thm]{Definition}
\newtheorem{alg}[thm]{Algorithm}
\newtheorem{prob}[thm]{Problem}
\newtheorem{rem}[thm]{Remark}

\renewcommand{\baselinestretch}{1.1}

\title{\bf On $r$-cross $t$-intersecting families for weak compositions}
\author{
Cheng Yeaw Ku
\thanks{ Department of Mathematics, National University of
Singapore, Singapore 117543. E-mail: matkcy@nus.edu.sg} \and Kok
Bin Wong \thanks{
Institute of Mathematical Sciences, University of Malaya, 50603
Kuala Lumpur, Malaysia. E-mail:
kbwong@um.edu.my.} } \maketitle

\begin{abstract}\noindent
Let $\mathbb N_0$ be the set of non-negative integers, and let $P(n,l)$ denote the set of all weak compositions of $n$ with $l$ parts, i.e., $P(n,l)=\{ (x_1,x_2,\dots, x_l)\in\mathbb N_0^l\ :\ x_1+x_2+\cdots+x_l=n\}$. For any element $\mathbf u=(u_1,u_2,\dots, u_l)\in P(n,l)$, denote its $i$th-coordinate by $\mathbf u(i)$, i.e., $\mathbf u(i)=u_i$. Let $l=\min(l_1,l_2,\dots, l_r)$. Families $\mathcal A_j\subseteq P(n_j,l_j)$ ($j=1,2,\dots, r$) are said to be $r$-cross $t$-intersecting if $\vert \{ i\in [l] \ :\ \mathbf u_1(i)=\mathbf u_2(i)=\cdots=\mathbf u_r(i)\} \vert\geq t$ for all $\mathbf u_j\in \mathcal A_j$. Suppose that $l\geq t+2$.   We prove that there exists a constant $n_0=n_0(l_1,l_2,\dots,l_r,t)$ depending only on $l_j$'s and $t$, such that for all $n_j\geq n_0$, if the families  $\mathcal A_j\subseteq P(n_j,l_j)$ ($j=1,2,\dots, r$) are $r$-cross $t$-intersecting, then
\begin{equation}
\prod_{j=1}^r \vert \mathcal{A}_j \vert\leq \prod_{j=1}^r {n_j+l_j-t-1 \choose l_j-t-1}.\notag
\end{equation}
Moreover, equality holds if and only if there is a $t$-set $T$ of $\{1,2,\dots,l\}$ such that $\mathcal{A}_j=\{\mathbf u\in P(n_j,l_j)\ :\ \mathbf u(i)=0\ \textnormal{for all}\ i\in T\}$ for $j=1,2,\dots, r$. 
\end{abstract}

\bigskip\noindent
{\sc keywords:}  cross-intersecting family, Erd{\H
o}s-Ko-Rado, weak compositions

\section{Introduction}

Let $[n]=\{1, \ldots, n\}$, and let ${[n] \choose k}$ denote the
family of all $k$-subsets of $[n]$. A family $\mathcal{A}$ of subsets of $[n]$ is $t$-{\em intersecting} if $|A \cap B| \ge t$ for all $A, B \in \mathcal{A}$. One of the most beautiful results in extremal combinatorics is
the Erd{\H o}s-Ko-Rado theorem.

\begin{thm}[Erd{\H o}s, Ko, and Rado \cite{EKR}, Frankl \cite{Frankl}, Wilson \cite{Wilson}]\label{EKR} Suppose $\mathcal{A} \subseteq {[n] \choose k}$ is $t$-intersecting and $n>2k-t$. Then for $n\geq (k-t+1)(t+1)$, we have
\begin{equation}
\vert \mathcal{A} \vert\leq {n-t \choose k-t}.\notag
\end{equation}
Moreover, if $n>(k-t+1)(t+1)$ then equality holds if and only if $\mathcal{A}=\{A\in {[n] \choose k}\ :\ T\subseteq A\}$ for some $t$-set $T$.
\end{thm}

In the celebrated paper \cite{AK}, Ahlswede and Khachatrian extended the Erd{\H o}s-Ko-Rado theorem by determining the structure of all  $t$-intersecting set systems of maximum size for all possible $n$ (see also \cite{Kee, Ku_Wong3, Toku} for some related results). There have been many recent results showing that a version of the Erd{\H o}s-Ko-Rado theorem holds for combinatorial objects other than set systems. For example, an analogue of the Erd{\H o}s-Ko-Rado theorem for the Hamming scheme is proved in \cite{Moon}. A complete solution for the $t$-intersection problem in the Hamming space is given in \cite{AK2}. Intersecting families of permutations were initiated by Deza and Frankl in \cite{DF}. Some recent work done on this problem and its variants can be found in \cite{B2, BH,CK, E, EFP, GM, KL, KW, LM, LW, WZ}. The investigation of the Erd{\H os}-Ko-Rado  property for graphs started in \cite{HT}, and gave rise to \cite{B, B3, HS, HST, HK, W}. The Erd{\H o}s-Ko-Rado type results also appear in vector spaces \cite{CP, FW},  set partitions \cite{KR, Ku_Wong, Ku_Wong2} and weak compositions \cite{Ku_Wong4}.

Let $\mathcal{A}_i\subseteq {[n] \choose k_i}$ for $i=1,2,\dots, r$. We say that the families $\mathcal{A}_1,\mathcal A_2,\dots, \mathcal A_r$ are $r$-\emph{cross} $t$-\emph{intersecting} if $\vert A_1\cap A_2\cap \cdots\cap A_r\vert\geq t$ holds for all $A_i\in\mathcal A_i$.  It has been shown by Frankl and Tokushige \cite{FT} that if $\mathcal{A}_1,\mathcal A_2,\dots, \mathcal A_r\subseteq {[n] \choose k}$ are $r$-cross 1-intersecting, then for $n\geq rk/(r-1)$,
\begin{equation}
\prod_{i=1}^r \vert \mathcal A_i\vert\leq \binom{n-1}{k-1}^r.\notag
\end{equation}

For different values of $k$'s, we have the following result. 

\begin{thm}[Bey \cite{Bey}, Matsumoto and Tokushige \cite{MT}, Pyber  \cite{Pyber}]\label{Bey} Let $\mathcal{A}_1\subseteq {[n] \choose k_1}$ and $\mathcal{A}_2\subseteq {[n] \choose k_2}$ be 2-cross 1-intersecting. If $k_1,k_2\leq n/2$, then
\begin{equation}
\vert \mathcal A_1\vert \vert \mathcal A_2\vert\leq \binom{n-1}{k_1-1} \binom{n-1}{k_2-1}.\notag
\end{equation}
Equality holds for $k_1+k_2<n$ if and only if $\mathcal A_1$ and $\mathcal A_2$ consist of all $k_1$-element resp.
$k_2$-element sets containing a fixed element.
\end{thm}

In this paper, we prove an analogue of Theorem \ref{Bey} for weak compositions with fixed number of parts. Let $\mathbb N_0$ be the set of non-negative integers, and let $P(n,l)$ denote the set of all weak compositions of $n$ with $l$ parts, i.e., $P(n,l)=\{ (x_1,x_2,\dots, x_l)\in\mathbb N_0^l\ :\ x_1+x_2+\cdots+x_l=n\}$. For any element $\mathbf u=(u_1,u_2,\dots, u_l)\in P(n,l)$, denote its $i$th-coordinate by $\mathbf u(i)$, i.e., $\mathbf u(i)=u_i$. Let $l=\min(l_1,l_2,\dots, l_r)$.  For any $\mathbf u_j\in P(n_j,l_j)$ ($j=1,2,\dots, r$), let $I(\mathbf u_1,\mathbf u_2,\dots, \mathbf u_r)=\{ i\in [l]\ :\ \mathbf u_1(i)=\mathbf u_2(i)=\cdots=\mathbf u_r(i)\}$. Families $\mathcal A_j\subseteq P(n_j,l_j)$ ($j=1,2,\dots, r$) are said to be $r$-\emph{cross} $t$-\emph{intersecting} if  $\vert I(\mathbf u_1,\mathbf u_2,\dots, \mathbf u_r)\vert\geq t$ for all $\mathbf u_j\in \mathcal A_j$. Our main result is the following.

\begin{thm}\label{thm_main} Given any positive integers $l_1,l_2,\dots, l_r$ and $t$ such that $l=\min(l_1,l_2,\dots, l_r)\geq t+2$, there exists a constant $n_0=n_0(l_1,l_2,\dots,l_r,t)$ depending only on $l_j$'s and $t$, such that for all $n_j\geq n_0$, if the families  $\mathcal A_j\subseteq P(n_j,l_j)$ \textnormal{($j=1,2,\dots, r$)} are $r$-cross $t$-intersecting, then
\begin{equation}
\prod_{j=1}^r \vert \mathcal{A}_j \vert\leq \prod_{j=1}^r {n_j+l_j-t-1 \choose l_j-t-1}.\notag
\end{equation}
Moreover, equality holds if and only if there is a $t$-set $T$ of $\{1,2,\dots,l\}$ such that $\mathcal{A}_j=\{\mathbf u\in P(n_j,l_j)\ :\ \mathbf u(i)=0\ \textnormal{for all}\ i\in T\}$ for $j=1,2,\dots, r$. 
\end{thm}

\section{Case $r=2$}

In this section, we will prove the following theorem which is a special case of Theorem \ref{thm_main}.

\begin{thm}\label{thm_pre_main} Given any positive integers $l_1,l_2$ and $t$ such that $l=\min(l_1,l_2)\geq t+2$, there exists a constant $n_0=n_0(l_1,l_2,t)$ depending only on $l_1,l_2$ and $t$, such that for all $n_1,n_2\geq n_0$, if the families  $\mathcal A_1\subseteq P(n_1,l_1)$ and $\mathcal A_2\subseteq P(n_2,l_2)$  are $2$-cross $t$-intersecting, then
\begin{equation}
\vert \mathcal{A}_1 \vert\vert \mathcal{A}_2 \vert\leq  {n_1+l_1-t-1 \choose l_1-t-1}{n_2+l_2-t-1 \choose l_2-t-1}.\notag
\end{equation}
Moreover, equality holds if and only if there is a $t$-set $T$ of $\{1,2,\dots,l\}$ such that 
\begin{align}
\mathcal{A}_1 & =\{\mathbf u\in P(n_1,l_1)\ :\ \mathbf u(i)=0\ \textnormal{for all}\ i\in T\}\ \textnormal{and}\notag\\
\mathcal{A}_2 & =\{\mathbf u\in P(n_2,l_2)\ :\ \mathbf u(i)=0\ \textnormal{for all}\ i\in T\}.\notag
\end{align}
\end{thm}

A family $\mathcal B\subseteq P(n,l)$ is said to be \emph{independent} if $I(\mathbf u,\mathbf v)=\varnothing$, i.e., $\vert I(\mathbf u,\mathbf v)\vert=0$, for all $\mathbf u,\mathbf v\in \mathcal B$ with $\mathbf u\neq \mathbf v$. We shall need the following theorem \cite[Theorem 2.3]{Ku_Wong4}

\begin{thm}\label{thm_independent} Let $m,n$ be positive integers satisfying $m\leq n$, and let $q,r,s$ be positive integers with $r,s\geq 2$ and $n\geq (2s)^{2^{r-2}q}+1$. If $\mathcal A\subseteq P(m,r)$ such that $\vert\mathcal A\vert\geq n^{\frac{1}{q}}\binom{n+r-2}{r-2}$,
then there is an independent set $\mathcal B\subseteq \mathcal A$ with $\vert \mathcal B\vert\geq s+1$.
\end{thm}

Let $\mathbf u=(u_1,u_2,\dots, u_l)\in P(n,l)$.  We define $R(i,\mathbf u)$ to be the element obtained from $\mathbf u$ by removing the $i$-th coordinate, i.e.,
\begin{equation}
R(i;\mathbf u)=(u_1,u_2,\dots, u_{i-1}, u_{i+1},\dots, u_l).\notag
\end{equation}
Inductively, if $x_1,x_2,\dots, x_t$ are distinct elements in $[l]$ with $x_1<x_2<\cdots<x_t$, we define
\begin{equation}
R(x_1,x_2,\dots, x_t;\mathbf u)=R(x_1,x_2,\dots, x_{t-1};R(x_t;\mathbf u)).\notag
\end{equation}
In other words, $R(x_1,x_2,\dots, x_t;\mathbf u)$ is the element obtained from $\mathbf u$ by removing the coordinates $x_{i}$, $ 1 \le i \le t$.

Let $\mathcal A\subseteq P(n,l)$. Let $x_1,x_2,\dots, x_t$ be distinct elements in $[l]$ with $x_1<x_2<\cdots<x_t$, and $y_1,y_2,\dots, y_t\in [n]$. We set
\begin{align}
\mathcal A(x_1,x_2,\dots, x_t;y_1,y_2,\dots, y_t) & =\{ \mathbf u\in \mathcal A\ :\ \mathbf u(x_i)=y_i\ \textnormal{for all $i$}\},\notag\\
\mathcal A^*(x_1,x_2,\dots, x_t;y_1,y_2,\dots, y_t) & =\{ R(x_1,x_2,\dots, x_t;\mathbf u)\ :\ \mathbf u\in \mathcal A(x_1,x_2,\dots, x_t;y_1,y_2,\dots, y_t)\}.\notag
\end{align}
Note that
\begin{equation}
\mathcal A^*(x_1,x_2,\dots, x_t;y_1,y_2,\dots, y_t)\subseteq P\left (n-\sum_{j=1}^t y_j,l-t\right ),\notag
\end{equation}
and
\begin{equation}
\vert \mathcal A^*(x_1,x_2,\dots, x_t;y_1,y_2,\dots, y_t)\vert=\vert \mathcal A(x_1,x_2,\dots, x_t;y_1,y_2,\dots, y_t)\vert.\notag
\end{equation}

\begin{lm}\label{lm_main_independent} Let $\mathcal A\subseteq P(n_1,l_1)$, $\mathbf v\in P(n_2,l_2)$ and  $l=\min(l_1,l_2)\geq t+2$. 
Suppose that $1\leq x_1<x_2<\cdots<x_t\leq l$. If $\mathcal A^*(x_1,\dots, x_t;y_1,\dots, y_t)$ has an independent set of size at least $l-t+1$, then either
\begin{itemize}
\item[\textnormal{(a)}] $\vert I(\mathbf v,\mathbf u)\vert \leq t-1$ for some $\mathbf u\in \mathcal A$, or
\item[\textnormal{(b)}] $\mathbf v\in \left \{ \mathbf u\in P(n_2,l_2)\ :\  \mathbf u(x_i)=y_i\ \textnormal{for all $i$}\right\}$.
\end{itemize}
\end{lm}

\begin{proof} Suppose that $\vert I(\mathbf v,\mathbf u)\vert \geq t$ for all $\mathbf u\in \mathcal A$ and $\mathbf v\notin \left \{ \mathbf u\in P(n_2,l_2)\ :\  \mathbf u(x_i)=y_i\ \textnormal{for all $i$}\right\}$. Let
\begin{equation}
\mathcal B=\{ R(x_1,x_2,\dots, x_t;\mathbf w_i)\ :\ i=1,2,\dots, l-t+1\},\notag
\end{equation}
be an independent set of size $l-t+1$ in $\mathcal A^*(x_1,\dots, x_t;y_1,\dots, y_t)$. 

Note that  $\mathbf w_i\in \mathcal A(x_1,\dots, x_t;y_1,\dots, y_t)$ for all $i$. Since $\mathbf v(x_j)\neq y_j$ for some $1\leq j\leq t$, we have 
\begin{equation}
\vert I(R(x_1,x_2,\dots, x_t;\mathbf v), R(x_1,x_2,\dots, x_t;\mathbf w_i))\vert\geq 1,\notag
\end{equation}
for all $i$. Let $\mathbf z=R(x_1,x_2,\dots, x_t;\mathbf v)$ and $\mathbf y_i=R(x_1,x_2,\dots, x_t;\mathbf w_i)$. Since $\mathcal B$ is independent,
\begin{equation}
I(\mathbf z,\mathbf y_i)\cap I(\mathbf z,\mathbf y_{i'})=\varnothing,\notag
\end{equation}
for $i\neq i'$. Therefore $\left\vert \bigcup_{i=1}^{l+t-1} I(\mathbf z,\mathbf y_i) \right\vert =\sum_{i=1}^{l-t+1} \vert I(\mathbf z,\mathbf y_i)\vert\geq \sum_{i=1}^{l-t+1} 1=l-t+1$. On the other hand, $\bigcup_{i=1}^{l+t-1} I(\mathbf z,\mathbf y_i)\subseteq [l] \setminus \{x_{j}: j \in [t]\}$, which is of size at most $l-t$, a contradiction. Hence, either part (a) or (b) of the lemma holds.
\end{proof}

\begin{proof}[Proof of Theorem \ref{thm_pre_main}] Let $\mathbf w=(w_1,w_2,\dots, w_{l_1})\in\mathcal A_1$  and $\mathbf v=(v_1,v_2,\dots, v_{l_2})\in\mathcal A_2$ be fixed. Since $\mathcal A_1$ and $\mathcal A_2$ are $2$-cross $t$-intersecting, we have
\begin{align}
\mathcal A_1 & =\bigcup_{\substack{\{x_1,x_2,\dots,x_t\}\subseteq [l],\\ x_1<x_2<\cdots<x_t}} \mathcal A_1(x_1,x_2,\dots, x_t;v_{x_1},v_{x_2},\dots, v_{x_t})\notag\\
\mathcal A_2 & =\bigcup_{\substack{\{x_1,x_2,\dots,x_t\}\subseteq [l],\\ x_1<x_2<\cdots<x_t}} \mathcal A_2(x_1,x_2,\dots, x_t;w_{x_1},w_{x_2},\dots, w_{x_t}).\notag
\end{align}

\vskip 0.5cm
\noindent
{\bf Case 1}. Suppose that
\begin{equation}
\vert \mathcal A_1(x_1,x_2,\dots, x_t;v_{x_1},v_{x_2},\dots, v_{x_t})\vert\leq n_1^{\frac{1}{2}}\binom{n_1+l_1-t-2}{l_1-t-2},\notag
\end{equation}
for all $\{x_1,x_2,\dots,x_t\}\subseteq [l]$ with $x_1<x_2<\cdots<x_t$. Then
\begin{align}
\vert \mathcal A_1\vert& \leq \binom{l}{t} n_1^{\frac{1}{2}}\binom{n_1+l_1-t-2}{l_1-t-2}\notag\\
&=\binom{l}{t} n_1^{\frac{1}{2}}\binom{n_1+l_1-t-1}{l_1-t-1}\frac{l_1-t-1}{n_1+l_1-t-1}\notag\\
& < \binom{l}{t} n_1^{-\frac{1}{2}}\binom{n_1+l_1-t-1}{l_1-t-1}(l_1-t-1).\notag
\end{align}

Now,
\begin{align}
\vert \mathcal A_2(x_1,x_2,\dots, x_t;w_{x_1},w_{x_2},\dots, w_{x_t})\vert & \leq \binom{n_2-\left(\sum_{i=1}^t w_{x_i}\right)+l_2-t-1}{l_2-t-1}\notag\\
&\leq \binom{n_2+l_2-t-1}{l_2-t-1},\notag
\end{align}
for all $\{x_1,x_2,\dots,x_t\}\subseteq [l]$ with $x_1<x_2<\cdots<x_t$. Therefore
\begin{equation}
\vert \mathcal A_2\vert \leq \binom{l}{t}\binom{n_2+l_2-t-1}{l_2-t-1},\notag
\end{equation}
and
\begin{equation}
\vert \mathcal{A}_1 \vert\vert \mathcal{A}_2 \vert\leq  \binom{l}{t}^2(l_1-t-1) n_1^{-\frac{1}{2}}\binom{n_1+l_1-t-1}{l_1-t-1}\binom{n_2+l_2-t-1}{l_2-t-1}.\notag
\end{equation}
Hence   $\vert \mathcal{A}_1 \vert\vert \mathcal{A}_2 \vert<\binom{n_1+l_1-t-1}{l_1-t-1}\binom{n_2+l_2-t-1}{l_2-t-1}$ if $n_1\geq \left((l_1-t-1)\binom{l}{t}^2\right)^2$.

\vskip 0.5cm
\noindent
{\bf Case 2}. Suppose that
\begin{equation}
\vert \mathcal A_2(x_1,x_2,\dots, x_t;w_{x_1},w_{x_2},\dots, w_{x_t})\vert\leq n_2^{\frac{1}{2}}\binom{n_2+l_2-t-2}{l_2-t-2},\notag
\end{equation}
for all $\{x_1,x_2,\dots,x_t\}\subseteq [l]$ with $x_1<x_2<\cdots<x_t$. This case is similar to Case 1. We will obtain $\vert \mathcal{A}_1 \vert\vert \mathcal{A}_2 \vert<\binom{n_1+l_1-t-1}{l_1-t-1}\binom{n_2+l_2-t-1}{l_2-t-1}$ if $n_2\geq \left((l_2-t-1)\binom{l}{t}^2\right)^2$.

\vskip 0.5cm
\noindent
{\bf Case 3}. Suppose that
\begin{align}
\vert \mathcal A_1(x_1,x_2,\dots, x_t;v_{x_1},v_{x_2},\dots, v_{x_t})\vert & \geq  n_1^{\frac{1}{2}}\binom{n_1+l_1-t-2}{l_1-t-2}\notag
\end{align}
for some $t$-set $\{x_1,x_2,\dots,x_t\}\subseteq [l]$ with $x_1<x_2<\cdots<x_t$, and 
\begin{align}
\vert \mathcal A_2(y_1,y_2,\dots, y_t;w_{y_1},w_{y_2},\dots, w_{y_t})\vert & \geq  n_2^{\frac{1}{2}}\binom{n_2+l_2-t-2}{l_2-t-2}\notag
\end{align}
for some $t$-set $\{y_1,y_2,\dots,y_t\}\subseteq [l]$ with $y_1<y_2<\cdots<y_t$.
Since
\begin{equation}
\vert \mathcal A_1^*(x_1,x_2,\dots, x_t;v_{x_1},v_{x_2},\dots, v_{x_t})\vert=\vert \mathcal A_1(x_1,x_2,\dots, x_t;v_{x_1},v_{x_2},\dots, v_{x_t})\vert,\notag
\end{equation}
and
\begin{equation}
\mathcal A_1^*(x_1,x_2,\dots, x_t;v_{x_1},v_{x_2},\dots, v_{x_t})\subseteq P\left (n-\sum_{j=1}^t v_{x_j},l_1-t\right ),\notag
\end{equation}
it follows from Theorem \ref{thm_independent} that $\mathcal A_1^*(x_1,x_2,\dots, x_t;v_{x_1},v_{x_2},\dots, v_{x_t})$ has an independent set of size at least $l-t+1$, provided
$n_1\geq (2(l-t))^{2^{l_1-t-1}}+1$. Similarly, $\mathcal A_2^*(y_1,y_2,\dots, y_t;w_{y_1},w_{y_2},\dots, w_{y_t})$ has an independent set of size at least $l-t+1$, provided
$n_2\geq (2(l-t))^{2^{l_2-t-1}}+1$.

It follows from Lemma \ref{lm_main_independent} that
\begin{align}
\mathcal A_1& \subseteq\left \{ \mathbf u\in P(n_1,l_1)\ :\  \mathbf u(y_i)=w_{y_i}\ \textnormal{for $i=1,2,\dots,t$}\right\}\ \textnormal{and}\notag\\
\mathcal A_2& \subseteq \left \{ \mathbf u\in P(n_2,l_2)\ :\  \mathbf u(x_i)=v_{x_i}\ \textnormal{for $i=1,2,\dots,t$}\right\}.\notag
\end{align}
This implies that
\begin{align}
\vert \mathcal A_1\vert &\leq \binom{n_1-\sum_{i=1}^t w_{y_i}+l_1-t-1}{l_1-t-1}\notag\\
 &\leq \binom{n_1+l_1-t-1}{l_1-t-1}.\notag
\end{align}
Clearly, equality holds if and only if $w_{y_i}=0$ for all $i=1,2,\dots, t$. Similarly, $\vert \mathcal A_2\vert \leq \binom{n_2+l_2-t-1}{l_2-t-1}$ and equality holds if and only if $v_{x_i}=0$ for all $i=1,2,\dots, t$. So $\vert \mathcal{A}_1 \vert\vert \mathcal{A}_2 \vert\leq \binom{n_1+l_1-t-1}{l_1-t-1}\binom{n_2+l_2-t-1}{l_2-t-1}$ and equality holds if and only if 
\begin{align}
\mathcal A_1& =\left \{ \mathbf u\in P(n_1,l_1)\ :\  \mathbf u(y_i)=0\ \textnormal{for $i=1,2,\dots,t$}\right\}\ \textnormal{and}\notag\\
\mathcal A_2& = \left \{ \mathbf u\in P(n_2,l_2)\ :\  \mathbf u(x_i)=0\ \textnormal{for $i=1,2,\dots,t$}\right\}.\notag
\end{align}
Now 
\begin{equation}
\mathcal A_1 =\mathcal A_1(y_1,y_2,\dots, y_t;\overbrace{0,0,\dots, 0}^t),\notag
\end{equation}
and 
\begin{equation}
\vert \mathcal A_1^*(y_1,y_2,\dots, y_t;\overbrace{0,0,\dots, 0}^t)\vert=\vert \mathcal A_1\vert=\binom{n_1+l_1-t-1}{l_1-t-1}>n_1^{\frac{1}{2}}\binom{n_1+l_1-t-2}{l_1-t-2}.\notag
\end{equation}
By Theorem \ref{thm_independent} and Lemma \ref{lm_main_independent}, 
\begin{equation}
\mathcal A_2 \subseteq \left \{ \mathbf u\in P(n_2,l_2)\ :\  \mathbf u(y_i)=0\ \textnormal{for $i=1,2,\dots,t$}\right\}.\notag
\end{equation}
Since $\vert \mathcal A_2\vert=\binom{n_2+l_2-t-1}{l_2-t-1}$, we must have $x_i=y_i$ for all $i$.

This completes the proof of Theorem \ref{thm_pre_main}.
\end{proof}

\section{Proof of Theorem \ref{thm_main}}

Note that $\mathcal A_i, \mathcal A_j$ are 2-cross $t$-intersecting for $i\neq j$. By Theorem \ref{thm_pre_main},
\begin{equation}\label{eq1}
\vert \mathcal A_i\vert \vert \mathcal A_j \vert\leq \binom{n_i+l_i-t-1}{l_i-t-1}\binom{n_j+l_j-t-1}{l_j-t-1}.\notag
\end{equation}

Therefore
\begin{align}
\left (\prod_{i=1}^r \vert \mathcal A_i\vert \right)^{r-1} & =\prod_{1\leq i< j\leq r}\vert \mathcal A_i\vert  \vert \mathcal A_j \vert\notag\\
& \leq \prod_{1\leq i< j\leq r}\binom{n_i+l_i-t-1}{l_i-t-1}\binom{n_j+l_j-t-1}{l_j-t-1}\notag\\
& = \left (\prod_{i=1}^r \binom{n_i+l_i-t-1}{l_i-t-1}\right)^{r-1},\notag
\end{align}
and 
\begin{equation}
\prod_{i=1}^r \vert \mathcal{A}_i \vert\leq \prod_{i=1}^r {n_i+l_i-t-1 \choose l_i-t-1}.\notag
\end{equation}
Suppose equality holds. Then
\begin{equation}
\vert \mathcal A_i\vert \vert \mathcal A_j \vert=\binom{n_i+l_i-t-1}{l_i-t-1}\binom{n_j+l_j-t-1}{l_j-t-1},\notag
\end{equation}
and by Theorem \ref{thm_pre_main}, there is a  $t$-set $T_{ij}$ of $\{1,2,\dots,l\}$ such that 
\begin{align}
\mathcal A_i& =\{\mathbf u\in P(n_i,l_i)\ :\ \mathbf u(x)=0\ \textnormal{for all}\ x\in T_{ij}\}\ \textnormal{and}\notag\\
\mathcal A_j& = \{\mathbf u\in P(n_j,l_j)\ :\ \mathbf u(x)=0\ \textnormal{for all}\ x\in T_{ij}\}.\notag
\end{align}

Suppose that $T_{ij}\neq T_{ij'}$ for some $j\neq j'$. Then
\begin{equation}
\mathcal A_i =\{\mathbf u\in P(n_i,l_i)\ :\ \mathbf u(x)=0\ \textnormal{for all}\ x\in T_{ij}\cup T_{ij'}\},\notag
\end{equation}
and $\vert \mathcal A_i\vert \leq \binom{n_i+l_i-t-2}{l_i-t-2}$. So,
\begin{equation}
\binom{n_i+l_i-t-1}{l_i-t-1}\binom{n_j+l_j-t-1}{l_j-t-1}=\vert \mathcal A_i\vert \vert \mathcal A_j \vert\leq \binom{n_i+l_i-t-2}{l_i-t-2}\binom{n_j+l_j-t-1}{l_j-t-1},\notag
\end{equation}
a contradiction. Hence $T=T_{ij}=T_{ij'}$ for all $j\neq j'$.  This completes the proof of Theorem \ref{thm_main}. \qed

\section*{Acknowledgments}	

This project is supported by the Advanced Fundamental Research Cluster, University of Malaya (UMRG RG238/12AFR).

\end{document}